\documentclass[10pt]{amsart}
\usepackage{amssymb,latexsym}
\theoremstyle{plain}
\newtheorem{theorem}{Theorem}

\newtheorem{lemma}{Lemma}

\newtheorem{proposition}{Proposition}
\newtheorem*{TA}{Theorem A}
\newtheorem*{TB}{Theorem B}

\theoremstyle{definition}
\newtheorem{definition}{Definition}

\theoremstyle{remark}

\numberwithin{equation}{section}
\newcommand{\R}{\mathbb R}

\newcommand{\acont}{\widetilde{A}}

\newcommand{\bcont}{\widetilde{B}}

\begin{document}

\title[The Benjamin-Ono equation]{The IVP for the Benjamin-Ono equation in weighted Sobolev spaces}
%%%%%%%%%%%%%%%%%%
%Author w_Nformatiw_N
%%%%%%%%%%%%%%%%%%%%
\author{Germ\'an Fonseca}
\address[G. Fonseca]{Departamento  de Matem\'aticas\\
Universidad Nacional de Colombia\\ Bogota\\Colombia}
\email{gefonsecab@unal.edu.co}
%\thanks{The second author is supported  by  the NSF  DMS-0800967}
%\thenks{}
%%%%%%%%%%%%%%%%%%%%%%
\author{Gustavo Ponce}
\address[G. Ponce]{Department  of Mathematics\\
University of California\\
Santa Barbara, CA 93106\\
USA.}
\email{ponce@math.ucsb.edu}
%\thenks{}
%%%%%%%%%%%%%%%%%%%%%%
%\date{}
%\dedicatory{}
%%%%%%%%%%%%%%
\begin{abstract} We study the initial value problem associated to the Benjamin-Ono equation.
The aim is to establish persistence properties of the solution flow in the weighted Sobolev spaces
$Z_{s,r}=H^s(\R)\cap L^2(|x|^{2r}dx)$, $s\in\R, \,s\geq 1$ and $s\geq r$. We also prove some  unique continuation properties of the solution flow in these spaces. In particular, these continuation principles demonstrate
that our persistence properties are sharp. 
\end{abstract}
\maketitle
%%%%%%%%%%%%%%%%%%
\begin{section}{Introduction}\label{S: 1}

This work is concerned with the initial value problem (IVP)
for the  Benjamin-Ono (BO) equation  
\begin{equation}\label{BO}
\begin{cases}
\partial_t u + \mathcal  H\partial_x^2u +u\partial_x u = 0, \qquad t, x
\in \R,\\
u(x,0) = u_0(x),
\end{cases}
\end{equation}
where $ \mathcal  H$ denotes  the Hilbert transform
\begin{equation}
\label{hilbertt}
\begin{aligned}
\mathcal H f(x)&=\frac{1}{\pi} \,\text{p.v.} (\frac{1}{x}\ast f)(x)\\
&=\frac{1}{\pi}\lim_{\epsilon\downarrow 0}\int_{|y|\geq \epsilon}\frac{f(x-y)}{y}dy=-i\,(\text{sgn}(\xi)\,\widehat {f}(\xi))^{\lor}(x).
\end{aligned}
\end{equation}

The BO equation  was deduced by Benjamin
\cite{Be} and Ono \cite{On} as  a model for long internal gravity waves in deep stratified fluids.
It was also shown that it is  a completely integrable system (see \cite{AbFo}, \cite{CoWi} and references therein).
 
Several works have been devoted to the problem of finding the minimal regularity, measured in the Sobolev scale
$ H^s(\R) = \left(1-\partial^2_x\right)^{-s/2} L^2(\R)$,  which guarantees that the IVP \eqref{BO} is locally or globally wellposed (LWP and GWP, resp.), i.e. existence and uniqueness  hold in a space embedded  in $C([0,T]:H^s(\R))$, 
with the map data-solution from $H^s(\R)$ to  $C([0,T]:H^s(\R))$ being locally continuous. Let us recall them : 
in \cite{Sa} $s>3$ was proven, 
in  \cite{ABFS} and 
\cite{Io1} $s>3/2$, 
in \cite{Po} $s\geq 3/2$, 
in \cite{KoTz1}  $s>5/4$, 
in \cite{KeKo} $s>9/8$,
in \cite{Ta} $s\geq 1$,
in \cite{BuPl} $s>1/4$,
and finally
in \cite{IoKe} $s\geq 0$ was established.

Real valued solutions of the IVP \eqref{BO} satisfy infinitely many conservation laws (time invariant quantities), the first three are the following:
\begin{equation}
\begin{aligned}
\label{laws}
&\;I_1(u)=\int_{-\infty}^{\infty}u(x,t)dx,\;\;\;\;I_2(u)=\int_{-\infty}^{\infty}u^2(x,t)dx,\\
&\;I_3(u)=\int_{-\infty}^{\infty}\,(|D_x^{1/2}u|^2-\frac {u^3}{3})(x,t)dx,
\end{aligned}
\end{equation}
where $D_x=\mathcal H\,\partial_x$.

Roughly, for $k\geq 2$ the $k$-conservation law $\,I_k$ provides an  \it a priori \rm estimate of the $L^2$-norm of the derivatives of order $(k-2)/2$  of the solution,  i.e. $\|D_x^{(k-2)/2}u(t)\|_2$. This  allows one to deduce GWP from LWP results. 

For existence of solutions with non-decaying at infinity initial data we refer to \cite{IoLiSc} and \cite{FoLi}. 

In the  BO equation the dispersive effect is described
by a non-local operator and   is significantly weaker than that exhibited by the 
Korteweg-de Vries (KdV) equation
$$
\partial_t u + \partial^3_x u + u \partial_x u = 0.
$$
Indeed,  it was proven in \cite{MoSaTz}  that for any $s\in\R$ the map data-solution from $H^s(\R)$ to $C([0,T]:H^s(\R))$ is not locally $C^2$, 
and in  \cite{KoTz2} that it is not locally uniformly continuous. This  implies that no LWP results can be obtained by an argument based only on a  contraction  method. This is certainly not the case of the KdV (see \cite{KePoVe}).

Our interest here is to study real valued solutions of the IVP \eqref{BO}  in weighted Sobolev spaces
\begin{equation}
\label{spaceZ}
Z_{s,r}=H^s(\R)\cap L^2(|x|^{2r}dx),\;\;\;\;\;\;s,\,r\in\R,
\end{equation}
and decay properties of solutions of the equation \eqref{BO}.
In this direction R. Iorio \cite{Io1} proved  the following results:

\begin{TA}\label{theorem1} $($\cite{Io1}$)$ (i) The IVP \eqref{BO} is GWP in $Z_{2,2}$.

(ii) If  $\,\widehat{u}_0(0)=0$, then the IVP \eqref{BO} is GWP 
in $\dot Z_{3,3}$.

(iii) If  $u(x,t)$ is a solution of the IVP \eqref{BO} such that $u\in C([0,T]: Z_{4,4}) $ for arbitrary  $T>0$, then $u(x,t)\equiv 0$.

\end{TA}

Above we have introduced the notation
\begin{equation}
\label{spaceZdot}
\dot Z_{s,r}=\{ f\in H^s(\R)\cap L^2(|x|^{2r}dx)\,:\,\widehat {f}(0)=0\},\;\;\;\;\;\;\;s,\,r\in\R.
\end{equation}
 Notice that the conservation law $I_1$ in  \eqref{laws} tells us that the property $\,\widehat{u}_0(0)=0$
 is preserved by the solution flow.
 
 We observe that the linear part of the equation in \eqref{BO} $L=\partial_t  + \mathcal  H\partial_x^2\,$
 commutes with the operator $\Gamma = x-2 t\mathcal H\partial_x$, i.e.
 $$
 [L;\Gamma]=L\Gamma-\Gamma L=0.
 $$
   In fact, one can deduce  (see \cite{Io1}) that for a solution $v(x,t)$ of the associated linear problem 
 \begin{equation}
 \label{linearasso}
 v(x,t)=U(t)v_0(x)=e^{-it\mathcal H\partial_x^2}v_0(x)=\ (e^{-it\xi|\xi|}\,\widehat{v}_0)^{\lor}(x),
 \end{equation}
to  satisfy that  
 $v(\cdot, t)\in  L^2(|x|^{2k}dx), \,t\in[0,T]$, one needs $v_0\in Z_{k,k},\;k\in \mathbb Z^+$ for $k=1,2$ and
  $$
  \int_{-\infty}^{\infty}x^j\,v_0(x)dx=0,\;\;\;\;j=0, 1,..., k-3,\;\;\;\text{if}\;\;\;k\geq 3.
  $$

   Also one notices that the 
  traveling wave $\phi_c(x+t),\,c>0$ for the BO equation 
  $$
  \phi(x)=\frac{-4\;\;\;\;}{1+x^2},\;\;\;\;\;\;\;\;\;\;\;\phi_c(x+t)=c\,\phi(c(x+ct)),
  $$
  has very mild decay at infinity. In this case, the traveling wave  is negative and travels to the left.
  To get a positive traveling wave moving to the right one needs to consider the equation
  \begin{equation}
  \label{BOneg}
  \partial_t v - \mathcal  H\partial_x^2v +v\partial_x v = 0, \qquad t, x\in\R,
  \end{equation}
  and observes that if $u(x,t)$ is a solution of  \eqref{BO} then
  $$
 v(x,t)=-u(x,-t),
 $$
 satisfies equation \eqref{BOneg}. In particular, \eqref{BOneg} has the traveling wave solution
 $$
 v(x,t)=\psi_c(x-t)= c\psi(c(x-ct)),\;\;\;\,c>0\;\;\;\;\;\text{with}\;\;\;\;\psi(x)=-\phi(x).
 $$
 
 In \cite{Io2} R. Iorio strengthened his unique continuation result  in  $Z_{4,4}$ found in  \cite{Io1} (Theorem A, part (iii)) by proving:
 
 \begin{TB}   $($\cite{Io2}$)$ \label{theorem2} 
 Let $u\in C([0,T] : H^2(\R))$ be a solution of the IVP \eqref{BO}. If   there exist  three different times
 $\,t_1, t_2, t_3\in [0,T]$ such that 
 \begin{equation}
 \label{3timesw} 
 u(\cdot,t_j)\in Z_{4,4},\;\;j=1,2,3,\;\;\text{then}\;\;\;u(x,t)\equiv 0.
 \end{equation}
 
\end{TB}

\vskip.1in

Our  goal in this work is to extend the results in Theorem A and Theorem B from integer values to the continuum  optimal range of indices $(s,r)$.
 Our main results are the following:
 
 \begin{theorem}\label{theorem3} (i) Let $s\geq 1, \;r\in [0,s]$, and $\,r<5/2$. If $u_0\in Z_{s,r}$, then the solution $u(x,t)$ of the IVP \eqref{BO} satisfies that $ u\in C([0,\infty):Z_{s,r})$.
 
 (ii) For  $s>9/8$  ($s\geq 3/2$), $\;r\in [0,s]$, and $\,r<5/2$ the IVP \eqref{BO} is LWP (GWP resp.) in $Z_{s,r}$.

(iii) If  $\,r\in [5/2,7/2)$ and $\,r\leq s$, then the IVP \eqref{BO} is GWP 
in $\dot Z_{s,r}$.

\end{theorem}

\begin{theorem}\label{theorem4} 
 Let $u\in C([0,T] : Z_{2,2})$ be a solution of the IVP \eqref{BO}. If   there exist  two different times
 $\,t_1, t_2\in [0,T]$ such that 
 \begin{equation}
 \label{2timesw} 
 u(\cdot,t_j)\in Z_{5/2,5/2},\;\;j=1,2,\;\;\text{then}\;\;\;\widehat {u}_0(0)=0\,,\;\;(\text{so}\;\; u(\cdot, t)\in  \dot Z_{5/2,5/2}).
 \end{equation}
 
\end{theorem}
 
 \vskip.1in
 
 \begin{theorem}\label{theorem5} 
 Let $u\in C([0,T] : \dot Z_{3,3})$ be a solution of the IVP \eqref{BO}. If   there exist  three different times
 $\,t_1, t_2, t_3\in [0,T]$ such that 
 \begin{equation}
 \label{3timeus} 
 u(\cdot,t_j)\in Z_{7/2,7/2},\;\;j=1,2,3,\;\;\text{then}\;\;\;u(x,t)\equiv 0.
 \end{equation}
 
\end{theorem}
\vskip.1in

\underline{Remarks} : (a) Theorem \ref{theorem4} shows that the condition $\widehat {u}_0(0)=0$ is necessary to have  persistence property
of the solution in $Z_{s,5/2}$,  with $s\geq 5/2$, so in that regard Theorem \ref{theorem3} parts (i)-(ii) are sharp. Theorem \ref{theorem5} affirms  that there is 
an upper limit of the spacial $L^2$-decay rate of the solution (i.e.  $|x|^{7/2}u(\cdot,t)\notin L^{\infty}([0,T]: L^2(\R))$, for 
any  $T>0$) regardless of the decay and  regularity of the non-zero initial data $u_0$. In particular, Theorem \ref{theorem5} shows that Theorem \ref{theorem3} part (iii) is sharp.

(b) In part (ii) of Theorem \ref{theorem3} we shall use that in that case 
the solution $u(x,t)$ satisfies
$$
\partial_x u\in L^1([0,T]:L^{\infty}(\R)),
$$
(see \cite{KeKo}, \cite{KoTz1}, and \cite{Po}) to  establish that the map data-solution is locally continuous from 
$Z_{s,r}$ into $C([0,T]:Z_{s,r})$.

 (c) The condition in Theorem \ref{theorem5} involving  three times seems to be technical and  may be reduced to  two different times as that  in Theorem \ref{theorem4} . We recall that unique continuation principles for the nonlinear Schr\"odinger equation and the generalized Korteweg-de Vries equation have been established in \cite{EKPV1} and \cite{EKPV2} resp. under assumptions on the solutions at two different times. 
 Following the idea in  \cite{Io2} one finds from  the equation \eqref{BO}  that 
 \begin{equation}
 \label{ole1}
 \frac{d\;}{dt}\int_{-\infty}^{\infty} x u(x,t) dx = \frac{1}{2}\,\|u(t)\|_2^2=\frac{1}{2}\,\|u_0\|_2^2,
\end{equation}
so the first momentum of a non-null solution of the BO equation  is strictly increasing. On the other hand, using the integral equation version of the BO equation 
from the hypotheses one can deduce that  the first momentum must vanish somewhere in the  time intervals $(t_1,t_2)$ and $(t_2,t_3)$. This implies  that $u(x,t)\equiv 0$.

(d) We recall that if  for a solution $u\in C([0,T]:H^s(\R))$ of \eqref{BO} one has that $\,\exists \,t_0\in[0,T]$ such that $u(x,t_0)\in H^{s'}(\R),\,s'>s$, then  $u\in C([0,T]:H^{s'}(\R))$. So we shall mainly consider the most interesting case $s=r$ in \eqref{spaceZ}.

(e) Consider  the IVP for generalized Benjamin-Ono (gBO) equation
\begin{equation}\label{gBO}
\begin{cases}
\partial_t u + \mathcal  H\partial_x^2u \pm u^k\partial_x u = 0, \qquad t, x
\in \R, k \in \mathbb Z^+,\\
u(x,0) = u_0(x),
\end{cases}
\end{equation}
with $u_0$ a real valued function. In this case the best LWP available results are : for $k=2, \,s\geq 1/2$ (see \cite{KeTa}), for $k=3,\, s>1/34$ (see \cite{Ve}), and 
for $k\geq 4,\,s\geq1/2-1/k$ (see \cite{Ve}). 
So for any power $k=1,2,...$ with focussing $(+)$ or defocusing $(-)$ non-linearity the IVP \eqref{gBO} is LWP in $H^1(\R)$. So the local results
 in  Theorems \ref{theorem3} and Theorem \ref{theorem4} and their proofs extend to the IVP \eqref{gBO}
 with possible different values $s=s(k)$ for the minimal regularity required. This is also the case for Theorem  \ref{theorem5} when the power $k$ in \eqref{gBO} is odd in the focusing and defocusing regime.
 
 (f) In \cite{KKNO} the number $7/2$ was mentioned as a possible threshold in the spaces \eqref{spaceZ}. 
 
\vskip.in

 The proof of Theorem \ref{theorem3} is based on weighted energy estimates and involves several inequalities for   the Hilbert transform $\mathcal H$.  
 Among them we shall use the $A_p$ condition introduced in \cite{Mu}, (see Definition \ref{definition1}). 
    It was proven in \cite{MuHuWh} that this is a necessary and sufficient condition  for the Hilbert transform $\mathcal H$ to be bounded in  
  $L^p(w(x)dx)$ (see  \cite{MuHuWh}, ), i.e. $\;w\in A_p,\;1<p<\infty$ if and only  if
\begin{equation}
\label{a1}
( \int_{-\infty}^{\infty}|\mathcal Hf|^pw(x)dx)^{1/p}\leq c^*\, (\int_{-\infty}^{\infty} |f|^pw(x)dx)^{1/p},
\end{equation}
(see Theorem \ref{theorem6}).

In order to justify some of our arguments in the  proofs we need some further  continuity properties of the Hilbert transform. More precisely, our  proof  requires
 the constant $c^*$ in \eqref{a1} to depend only on  $c(w)$  the constant describing the $A_p$ condition
(see \eqref{ap}) and on  $p$. In \cite{Pe} precise bounds for the constant $c^*$ in \eqref{apb} were given
which are sharp in  the case $p=2$ and sufficient for our purpose (see Theorem \ref{theorem7}). 
 
 It will be essential in  
our arguments that some commutator operators involving the Hilbert transform $\mathcal H$  are of   ``order zero".
More precisely, we shall use  the following estimate: 
 $\forall \,p\in(1,\infty),$ $l,\,m\in\mathbb  Z^+\cup\{0\},\,l+m\geq 1$  
$ \,\exists\, c=c(p;l;m)>0$ such that
\begin{equation}
\label{77}
 \| \partial_x^l[\mathcal H;\,a]\partial_x^m f\|_p\leq c \|\partial_x^{l+m} a\|_{\infty} \|f\|_p.
\end{equation}
 In the
case $l+m=1$,
\eqref{77} is Calder\'on's first commutator estimate
\cite{Ca}. In the  case $l+m\geq 2$,  \eqref{77} was proved in \cite{DaMcPo}.

 The rest of this paper is organized as follows: section 2 contains some preliminary estimates  to be utilized in  the coming sections.
 Theorem \ref{theorem3} will be proven in section 3. Finally, the proofs of  Theorem \ref{theorem4}
 and Theorem \ref{theorem5} 
 will be given in sections 4 and 5, respectively. 
 \end{section}

%%%%%%%%%
\begin{section}{Preliminary Estimates}\label{S: 2}

 We shall use the following notations:
 
 \begin{equation}
 \label{notation}
 \begin{aligned}
& \|f\|_p=(\int_{\infty}^{\infty}|f(x)|^pdx)^{1/p},\;\;\;1\leq p<\infty,\;\;\;\;\;\;\|f\|_{\infty}=\sup_{x\in\R}|f(x)|,\\
&  \|f\|_{s,2}=\|(1-\partial_x^2)^{s/2}f\|_2,\;\;\;\;s\in\R.
\end{aligned}
\end{equation}
 
 Let us first recall the definition of the $A_p$ condition. We shall restrict here to the cases $p\in(1,\infty)$ and the 1-dimensional case $\R$ (see \cite{Mu}).
 
 \begin{definition}\label{definition1} A non-negative function $w\in L^1_{loc}(\R)$ satisfies the $A_p$ inequality with $1<p<\infty\,$  if
 \begin{equation}
 \label{ap}
 \sup_{Q\;\text{interval}}\left(\frac{1}{|Q|}\int_Q w\right)\left(\frac{1}{|Q|}\int_Qw^{1-p'}\right)^{p-1}=c(w)<\infty,
 \end{equation}
 where $1/p+1/p'=1$.
 \end{definition}

  \begin{theorem}\label{theorem6}$($\cite{MuHuWh}$)$  The condition \eqref{ap} is necessary and  sufficient  for the boundedness of the Hilbert transform $\mathcal H$  in $L^p(w(x)dx)$, i.e.
 \begin{equation}
 \label{apb}
( \int_{-\infty}^{\infty}|\mathcal Hf|^pw(x)dx)^{1/p}\leq c^*\, (\int_{-\infty}^{\infty} |f|^pw(x)dx)^{1/p}.
 \end{equation}
 \end{theorem}

In the case $p=2$, a previous characterization of $w$ in \eqref{apb} was found in \cite{HeSz} (for further references and comments we refer to  \cite{Do}, \cite{GaRu}, and \cite{St2}). However, even though  we will be mainly concerned with the case $p=2$, the characterization \eqref{apb} will be the one used in our proof.
In particular, one has that in $\R$ 
\begin{equation}
\label{cond|x|}
|x|^{\alpha}\in A_p\;\;\Leftrightarrow\;\; \alpha\in (-1,p-1).
\end{equation}

In order to justify some of the arguments in the  proof of Theorem \ref{theorem3} we need some further  continuity properties of the Hilbert transform. More precisely, our  proof  requires
 the constant $c^*$ in \eqref{apb} to depend only on $c(w)$  in \eqref{ap}
 and on  $p$ (in fact,  this is only needed  for the case $p=2$). 
 
  \begin{theorem}\label{theorem7}$($\cite{Pe}$)$ For $p\in[2,\infty)$ the inequality \eqref{apb} holds with  $c^*\leq c(p)c(w)$, with $c(p)$ depending only on  $p$ and  $c(w)$ as in \eqref{ap}.  Moreover, for $p=2$ this estimate is sharp.
 \end{theorem}

  Next, we define the  truncated weights $w_N(x)$  using the notation  $\langle x\rangle =(1+x^2)^{1/2}$ as
\begin{equation}
\label{truncw}
 w_N(x)=
\begin{cases}
\langle x
\rangle& \text{if\,\,$|x|\le N$,}\\
2N  &  \text{if\,\,$|x|\ge 3N,$}
\end{cases}
\end{equation}
 $w_N(x) $ are smooth and non-decreasing in $|x|$ with  $w_{N}'(x)\leq 1$ for all $x\geq 0$.
  
  \begin{proposition}\label{proposition1} 
  For any $\,\theta\in (-1,1)$ and any $N\in \mathbb  Z^+$, $w^{\theta}_N(x)$ satisfies the $A_2$ inequality \eqref{ap}. Moreover, the Hilbert transform $\mathcal H$ 
 is bounded in $L^2(w^{\theta}_N(x)dx)$ with a constant depending on  $\theta$ but independent   of $N\in \mathbb Z^+$.
 \end{proposition}
 
 The proof of Proposition \ref{proposition1} follows by combining the fact that 
 for a fixed $\theta\in(-1,1)$ the family of weights $\,w^{\theta}_N(x),\,N\in \mathbb Z^+$ satisfies the $A_2$ inequality  in \eqref{ap} with a constant $c$  independent of $N$, and Theorem \ref{theorem7}.

 Next, we have the following generalization of Calder\'on commutator estimates \cite{Ca} found in \cite{DaMcPo} and 
already commented in the introduction:
 
   \begin{theorem}\label{theorem8}  
  For any $ \,p\in(1,\infty)$  and $l,\,m\in\mathbb  Z^+\cup\{0\},\,l+m\geq 1$  there exists 
$\, c=c(p;l;m)>0$ such that
 \begin{equation}
\label{777}
 \| \partial_x^l[\mathcal H;\,a]\partial_x^m f\|_p\leq c \|\partial_x^{l+m} a\|_{\infty} \|f\|_p.
\end{equation}
\end{theorem}
\vskip.1in

We shall also use the pointwise identities
$$
[\mathcal H; x]\partial_x f = [\mathcal H; x^2]\partial_x^2 f=0,
$$
and more generally
$$
[\mathcal H; x] f =0\;\;\;\;\text{if and only if }\;\;\;\;\int fdx=0.
$$
 We recall the following   characterization of the $L^p_s(\R^n)=(1-\Delta)^{-s/2}L^p(\R^n)$ spaces given in \cite{St1}.

 \begin{theorem}
 \label{theorem9}\label{St1} 
Let $b\in (0,1)$ and $\; 2n/(n+2b)< p< \infty$. Then $f\in  L^p_b(\R^n)$ if and only if
\begin{equation}
\label{d1}
\begin{aligned}
&\;(a)\;\, f\in L^p(\R^n),\\
\\
&\;(b)\;\;\;\;\;\mathcal D^b f(x)=(\int_{\R^n}\frac{|f(x)-f(y)|^2}{|x-y|^{n+2b   }}dy)^{1/2}\in L^p(\R^n),
\end{aligned}
\end{equation}
with
 \begin{equation}
\label{d1-norm}
\|f\|_{b,p}\equiv  \|(1-\Delta)^{b/2} f\|_p=\|J^b f\|_p\simeq \|f\|_p+\|D^b    f\|_p\simeq \|f\|_p+\|\mathcal D^b       f\|_p.
\end{equation}
 \end{theorem}

Above we have used the notation: for $s\in\R$
$$
D^s=(-\Delta)^{s/2}\;\;\;\;\;\text{with}\;\;\;\;D^s=(\mathcal H\,\partial_x)^s,\;\;\;\text{if}\;\;\;n=1.
$$

For the proof of this theorem we refer the reader to  \cite{St1}.
  One sees that from \eqref{d1} for $p=2$ and $b\in(0,1)$ one has
\begin{equation}
\label{pointwise2}
\|  \mathcal D^b      (fg)\|_2\leq \|f\, \mathcal D^b      g\|_2  +\|g\, \mathcal D^b      f\|_2.
\end{equation}
We shall use this estimate in the proof of Theorem \ref{theorem5}. As   applications of Theorem \ref{theorem9} 
we have the following estimate:

 \begin{proposition}
\label{propositionB}
 Let $b \in (0,1)$. For  any $t>0$
 \begin{equation}
 \label{pointwise-es}
  \mathcal D^b      (e^{-it x |x|})\leq c(|t|^{b/2}+|t|^b    |x|^b   ).
  \end{equation}
 \end{proposition}

For the proof of Proposition \ref{propositionB} we refer to \cite{NaPo}.

As a further  direct  consequence of Theorem \ref{theorem9}  we deduce the following result to be used in the proof of
 Theorem \ref{theorem5}.
  \begin{proposition}
 \label{proposition2}
 Let $\,p\in(1,\infty)$. If $f\in L^p(\R)$  such that  there exists $x_0\in\R$ for
 which $f(x_0^+),\,f(x_0^-)$ are defined and 
 $\,f(x_0^+)\neq f(x_0^-)$, then for any $\delta>0$, $\mathcal{D}^{1/p}\,f\notin L^p_{\it loc}(B(x_0,\delta))$ and consequently $\,f\notin L^p_{1/p}(\R)$.
 \end{proposition}
 
Also as consequence of  the estimate \eqref{pointwise2} one has  the following 
interpolation inequality.

\begin{lemma}
\label{lemma1}
Let $a,\,b>0$. Assume that $ J^af=(1-\Delta)^{a/2}f\in L^2(\mathbb R)$ and \newline $\langle x\rangle^{b}f=
(1+|x|^2)^{b/2}f\in L^2(\mathbb R)$. Then for any $\theta \in (0,1)$
\begin{equation}
\label{complex}
\|J^{ \theta a}(\langle x\rangle^{(1-\theta) b} f)\|_2\leq c\|\langle x\rangle^b f\|_2^{1-\theta}\,\|J^af\|_2^{\theta}.
\end{equation}
 Moreover, the inequality \eqref{complex} is still valid with $w_N(x)$ in \eqref{truncw} instead of $\langle x\rangle$ with a
constant $c$ independent of $N$.
\end{lemma}

 \begin{proof}
It will suffice to consider the case : $a=1+\alpha,\,\alpha\in (0,1)$. We denote by $\rho(x)$ a function equal to $\langle x\rangle$ or equal to $w_N(x)$ as in \eqref{truncw} and consider the function
$$
F(z)=\;e^{(z^2-1)} \int_{-\infty}^{\infty} \,J^{az}(\rho ^{b(1-z)} f(x) )\,\overline{g(x)}dx
$$
with $\,g\in L^2(\mathbb R^n)$ with $\|g\|_2=1$, which is continuous  in $\{z=\eta+iy\,:\,0\leq \eta\leq 1\}$ and analytic in its interior.  Moreover,
$$
|F(0+iy)|\leq e^{-(y^2+1)} \,\|\rho^b\,f\|_2,
$$
and since $|\rho'/\rho|+|\rho''/\rho|\leq c$ (independent of $N$) combining \eqref{d1} and \eqref{pointwise2} one has 
$$
\aligned
&|F(1+iy)|\leq e^{-y^2} \|  J^a(\rho^{iby} f)\|_2 \leq e^{-y^2}(\|\rho^{iby} f\|_2 +\| D^{\alpha}\partial_x(\rho^{iby} f)\|_2)\\
&\leq e^{-y^2}(\| f\|_2 +\| D^{\alpha} (\rho^{iby}\partial_x f)\|_2 + |by| \,\| D^{\alpha} (\rho^{iby-1}\rho'\, f)\|_2)\\
&\leq e^{-y^2}(\| f\|_2 +\| \mathcal D^{\alpha} (\rho^{iby}\partial_x f)\|_2 + |by| \,\| \mathcal D^{\alpha} (\rho^{iby-1}\rho'\, f)\|_2)\\
&\leq e^{-y^2}(\| f\|_2 +\| \mathcal D^{\alpha} (\rho^{iby}) \partial_x f\|_2 + \|  (\rho^{iby}) \mathcal D^{\alpha}\partial_x f \|_2\\
&\;\; \;\;\;+
|by| \,\| \mathcal D^{\alpha} (\rho^{iby-1}\rho')\, f \|_2 +
|by| \,\|  (\rho^{iby-1}\rho')\,  \mathcal D^{\alpha} f \|_2)\\
&\leq c_{\alpha}\,e^{-y^2} \,(1+|yb|^2)(\|f\|_2 + \|\mathcal D^{\alpha} f\|_2+\|\partial_x f\|_2 +\| \mathcal D^{\alpha}\partial_x f\|_2)\\ 
&\leq c_{\alpha}\,e^{-y^2} \,(1+|yb|^2)\|J^{1+\alpha}f\|_2= c_{\alpha}\,e^{-y^2} (1+|yb|^2) \|J^{a}f\|_2,
\endaligned
$$
using that for $\alpha\in(0,1)$
$$
\|\mathcal D^{\alpha} h\|_{\infty}\leq c_{\alpha}(\|h\|_{\infty}+\|\partial_xh\|_{\infty}).
$$
Therefore, the three lines theorem yields the desired result.
\end{proof}

 We shall also employ the following simple estimate.
 
 \begin{proposition}
 \label{proposition3}
 If $f\in L^2(\R)$ and $\phi\in H^1(\R)$, then
 \begin{equation}
 \label{simplein}
 \| [D^{1/2};\phi] f\|_2\leq c\|\phi\|_{1,2} \|f\|_2.
 \end{equation}
 \end{proposition}

 Finally, to complete this section we recall the result obtained in \cite{Po} 
 concerning  regularity properties of the solutions of the IVP \eqref{BO} 
 with data $u_0\in H^s(\R),\,s\geq 3/2$. This will be used in the proof of Theorem \ref{theorem5}.
 
 \begin{theorem}
 \label{theorem10}
For any $u_0\in H^s(\R)$ with $s\geq 3/2$ the IVP \eqref{BO} has a unique global solution $u\in
C([0,T]:H^s(\R))$ such that for any $T>0$ 
$$
J^{s+1/2} u \in l^{\infty}_k(L^2([k,k+1]\times[0,T])),\;\;Ju\in l^{2}_k(L^{\infty}([k,k+1]\times[0,T]))
$$ 
and
$$
J^{s-3/2}\partial_x u\in L^4([0,T]:L^{\infty}(\R)).
$$

\end{theorem}

 \end{section}

 \begin{section}{Proof of Theorem \ref{theorem3} }\label{S: 3}
 
 We consider several cases :
 
 \underline{Case 1:  $s=1$ and $r=\theta\in(0,1]$}. Part (i) in Theorem \ref{theorem3}:

 We multiply the differential equation by $w_N^{2\theta} u$ 
 (see \eqref{truncw}) with $ 0<\theta\le1$ and integrate  on $\R$ to obtain
\begin{equation}
\label{eq2}
\frac{1}{2}\frac{d}{dt}\int{(w_N^{\theta}u)^2\,dx}+\int{w_N^{\theta}\mathcal H \partial_x^2u \,w_N^{\theta}u\,dx}+\int{w_N^{2\theta}u^{2}\partial_x u\,dx}=0.
\end{equation}

 To handle the second term on the  left hand side (l.h.s.) of \eqref{eq2} we write
$$
\aligned
 w_N^{\theta}\mathcal H \partial_x^2u &= [ w_N^{\theta};\mathcal H] \partial_x^2u +
\mathcal H(w_N^{\theta}\partial_x^2u)\\
&= A_1 + \mathcal H\partial_x^2(w_N^{\theta}u)
-2 \mathcal H(\partial_xw_N^{\theta} \partial_xu) -\mathcal H(\partial^2_xw_N^{\theta} u)\\
&=A_1+A_2+A_3+A_4.
\endaligned
$$
We observe that by Theorem \ref{theorem8} and our assumption on $\theta\in(0,1]$ the terms 
$A_1, A_4$ are bounded by the $L^2$-norm  of the solution $u$  and $A_3$ is bounded by the $H^1$-norm of the solution with constants independent of $N$, thus they are bounded uniformly on  $N\in \mathbb Z^+$ by
$$
M_1=\sup_{t\in[0,T]}\|u(t)\|_{1,2}.
$$
We insert the  term $A_2$ in \eqref{eq2} and use  integration by parts, to get that
$$
\int \mathcal H\partial_x^2(w_N^{\theta}u) w_N^{\theta} u dx=0.
$$
Finally, using integration by parts, we bound the nonlinear term (the third term on the l.h.s.) in \eqref{eq2} as
\begin{equation}
\label{nonlinear1}
| \int w_N^{2\theta}u^{2}\partial_x u\,dx|\leq c\|u\|_{\infty}\|u\|_2 \|w_N^{\theta} u\|_2
\leq c \|u\|^2_{1,2} \|w_N^{\theta} u\|_2.
\end{equation}

Inserting this information in \eqref{eq2} we get
$$
\frac{d}{dt}\|w_N^{\theta}u(t)\|_2 \leq cM,\;\;\;\;\text{with}\;\;c\;\text{independent of}\;\;N,
$$
which tells us that
$$
\sup_{t\in[0,T]}\| w_N^{\theta}u(t)\|_2\leq c\|\langle x\rangle^{\theta} u_0\|_2\,e^{TM},\;\;\;\;\text{with}\;\;c\;\text{independent of}\;\;N,
$$
which yields the result $u\in  L^{\infty}([0,T]: L^2(|x|^{2\theta}))$ for any $T>0$. 

To see that  $u\in  C([0,T]: L^2(|x|^{2\theta}))$ 
one considers the sequence 
$$
( w_N^{\theta}u)_{N\in\mathbb Z^+}\subseteq C([0,T]: L^2(\R)),
$$
 and reapply the above argument to find that it is a Cauchy sequence.
 
 Finally, we point out that the use of the differential equation in \eqref{BO} can be justified by  the locally continuous dependence of the solution upon the data from $H^s(\R)$ to $C([0,T] : H^s(\R))$.
 
\vskip.1in
%%%%%%%%%%%%%%%%%%%%%%%%%%%%%%%%%%%%
 \underline{Case 2:  $s\in (1,2]$ and $r=s$}.  Part (i) in Theorem \ref{theorem3}.
 
 We multiply the differential equation by $w_N^{2+2\theta} u$ 
 (see \eqref{truncw}) with $ 0\le\theta\le1$ and integrate  on $\R$ to obtain
\begin{equation}
\label{eq4}
\frac{1}{2}\frac{d}{dt}\int{(w_N^{1+\theta}u)^2\,dx}+\int{w_N^{1+\theta}\mathcal H \partial_x^2u \,w_N^{1+\theta}u\,dx}+\int{w_N^{2+2\theta}u^{2}\partial_x u\,dx}=0.
\end{equation}

To control the second term on the  l.h.s. of \eqref{eq4} we write
$$
\aligned
 w_N^{1+\theta}\mathcal H \partial_x^2u &= [ w_N^{1+\theta};\mathcal H] \partial_x^2u +
\mathcal H(w_N^{1+\theta}\partial_x^2u)\\
&= B_1 + \mathcal H\partial_x^2(w_N^{1+\theta}u)
-2 \mathcal H(\partial_xw_N^{1+\theta} \partial_xu) -\mathcal H(\partial^2_xw_N^{1+\theta} u)\\
&=B_1+B_2+B_3+B_4.
\endaligned
$$
We observe that by Theorem \ref{theorem8} and our assumption  $\theta\in(0,1]$ the terms 
$B_1, B_4$ are bounded by the $L^2$-norm  of the solution. Inserting the term $B_2$ on the  
\eqref{eq4} and using integration by parts one finds that its contribution is null. So it remains to control $B_3=
-2 \mathcal H(\partial_xw_N^{1+\theta} \partial_xu)$. Since
$$
|\partial_x w_{N}^{1+\theta}|=|(1+\theta)w_N^{\theta}\,\partial_x w_N| \leq c\,w_N^{\theta},\;\;\;\;c\;\;\text{independent of }\;\;N,
$$
one has
\begin{equation}
\label{22}
\|B_3\|_2\leq c\,\|\,w_N^{\theta} \partial_x u\|_2 \leq c  \|\partial_x(w_N^{\theta} u)\|_2
+ c \| \partial_x w_N^{\theta}  u\|_2 \leq c  \|\partial_x(w_N^{\theta} u)\|_2 +c \|u\|_2.
 \end{equation}
 Then by the  interpolation inequality in \eqref{complex} it follows that
\begin{equation}
\label{int1}
\|\partial_x(w_N^{\theta} u)\|_2 \leq \| J(w_N^{\theta}  u)\|_2
\leq c \| w_{N}^{1+\theta}u\|_2^{\theta/(1+\theta)}    \|J^{1+\theta}u\|_2^{1/(1+\theta)},
\end{equation}
with a constant $c$ independent of $N$. So by Young's inequality  in \eqref{int1} and \eqref{22} the term $B_3$ is 
appropriately bounded. Finally, for the last term on the l.h.s. of \eqref{eq4} we write
\begin{equation}
\label{nonlinear2}
| \int w_N^{2+2\theta}u^{2}\partial_x u\,dx|\leq c\|u\|_{\infty} \|w_N^{1+\theta} u\|_2^2
\leq c\|u\|_{1,2} \|w_N^{1+\theta} u\|^2_2,
\end{equation}
with $c$ independent of $N$.

So inserting the above information in \eqref{eq4} we obtain the result.

\vskip.1in

%%%%%%%%%%%%%%%%%%%%%%%%%%%%%%
 \underline{Case 3: $s\in (9/8,2]$ and $r=s$}.  Part (ii)  in Theorem \ref{theorem3}.
 
 In this case it remains to establish the continuous dependence of the solution $C([0,T]:Z_{s,r})$ upon the data in $Z_{s,r}$. We are considering the most interesting case $s=r\in(9/8,2]$. Suppose that $u,\,v\in C([0,T]:Z_{s,s})$
 are two solutions of the BO equation in \eqref{BO} corresponding to data $u_0,\,v_0$ respectively.
 Hence,
 \begin{equation}
 \label{dossol}
 \partial_t(u-v) +\mathcal H \partial_x (u-v) + \partial_xu (u-v) +v \partial_x(u-v)=0.
 \end{equation}
 We will reapply the argument used in the previous case. However, we notice that the nonlinear term in \eqref{dossol}
 is different than that in \eqref{eq4}.
 So we recall the result in \cite{KeKo} which affirms that for $s>9/8$
\begin{equation} 
\label{est1inf}
\partial_xu,\,\partial_xv\in L^1([0,T]:L^{\infty}_x(\R)),
\end{equation}
and use integration by parts to obtain that
\begin{equation}
\label{nonlineardiff}
\begin{aligned}
&| \int w_N^{2+2\theta}(\partial_xu (u-v)^2 +v (u-v) \partial_x(u-v))\,dx|\\
&\leq 
c(\|\partial_x u(t) \|_{\infty} +\|\partial_x v(t)\|_{\infty}+\|v(t)\|_{\infty})\, \|w_N^{1+\theta}(u-v)\|_2^2.
\end{aligned}
\end{equation}

Hence,  combining  the argument in the previous section, the estimates \eqref{nonlineardiff} 
and \eqref{est1inf}, and the continuous dependence of the solution in $C([0,T]: H^s(\R))$ upon the 
data in $H^s(\R)$ the desired result follows. 

 \vskip.1in
%%%%%%%%%%%%%%%%%%%%%%%%%%%%%%%%%%%%%%%
 
  \underline{Case 4: $s=r\in (2,5/2)$}.  Part (ii) in Theorem \ref{theorem3}. 
  
 We recall that from the previous cases we know the result for $s\geq r\in (0,2]$.
 Also we shall write $r=2+\theta,\,\,\theta\in(0,1/2)$, and we multiply the differential equation by $x^2 w_N^{2+2\theta} u$ 
 (see \eqref{truncw})  and integrate  on $\R$ to obtain
\begin{equation}
\label{eq5}
\begin{aligned}
&\frac{1}{2}\frac{d}{dt}\int{( w_N^{1+\theta}  x u)^2\,dx}\\
&+\int{w_N^{1+\theta}\,x\,\mathcal H \partial_x^2u \,w_N^{1+\theta}\,x\,u\,dx}+\int{x^2 w_N^{2+2\theta}u^{2}\partial_x u\,dx}=0.
\end{aligned}
\end{equation}
From our previous proofs it is clear that we just need to handle the second term on the l.h.s. of \eqref{eq5}.
First we write the identity 
\begin{equation}
\label{iden1}
x\,\mathcal H \partial_x^2 u=\mathcal H(x\partial_x^2 u)=\mathcal H(\partial_x^2(xu)) - 2\mathcal H\partial_xu
=E_1+E_2.
\end{equation}
  
  To bound the contribution of the term $E_2$ inserted in \eqref{eq5} we shall use that $w_N^{\theta}$ with $\theta\in(0,1/2)$ satisfies the $A_2$ inequality 
  uniformly in $N$ (see Proposition \ref{proposition1}) so
  \begin{equation}
  \label{eq7}
  \begin{aligned}
\|w_N^{1+\theta} E_2\|_2&=2 \| w_N^{1+\theta}\,\mathcal H\partial_x u\|_2\leq  c\| w_N^{\theta}\,\mathcal H\partial_x u\|_2
+ c\| w_N^{\theta}\,x\,\mathcal H\partial_x u\|_2\\
&  \leq  c\| w_N^{\theta}\,\partial_x u\|_2
+ c\| w_N^{\theta}\,\mathcal H(x\partial_x u)\|_2\\
&  \leq  c\| w_N^{\theta}\,\partial_x u\|_2
+ c\| w_N^{\theta}\,x \,\partial_x u\|_2=F_1+F_2.
\end{aligned}
\end{equation}
  
  Now using complex interpolation one gets (see Lemma \ref{lemma1})
  \begin{equation}
  \label{eq8}
  \begin{aligned}
 \| w_N^{\theta}\,\partial_x u\|_2&\leq \|\partial_x(w_N^{\theta} u)\|_2 +\|\partial_x w_N^{\theta} u\|_2\\
 &\leq 
   \|\partial_x(w_N^{\theta} u)|_2 + c \| u\|_2 \leq c||J(w_N^{\theta} u)\|_2 +c\|u\|_2\\
   &\leq 
 c  \| J^2u\|_2^{1/2}\| \langle x\rangle^{2 \theta} u\|_2^{1/2} + c \|u\|_2,
\end{aligned}
  \end{equation}
  which has been  bounded in the previous cases. 
  So it remains to bound the term
  \begin{equation}
  \label{remain1}
  F_2=\| w_N^{\theta}\,x \,\partial_x u\|_2,
  \end{equation}
  which will be considered later.
  
   Inserting the term $E_1$ in \eqref{iden1} into \eqref{eq5} one 
   obtains the term
   \begin{equation}
   \label{eq10}
G_1=   \int{w_N^{1+\theta}\,\mathcal H \partial_x^2(xu) \,w_N^{1+\theta}\,x\,u\,dx}.
\end{equation}
As before we write
\begin{equation}
\begin{aligned}
   \label{eq11}
&w_N^{1+\theta}\,\mathcal H \partial_x^2(xu) = -[\mathcal H; w_N^{1+\theta}] \partial_x^2(xu)
+\mathcal H(w_N^{1+\theta} \partial_x^2(xu)) \\
& = K_1+ \mathcal H(\partial_x^2(w_N^{1+\theta}   x u))-2\mathcal H(\partial_x w_N^{1+\theta} \partial_x(xu))- \mathcal H(\partial_x^2w_N^{1+\theta}(xu))\\
&=K_1+K_2+K_3+K_4.
\end{aligned}
\end{equation}

Thus, by Theorem \ref{theorem8} and the results in the previous cases the contribution of $K_1, \,K_4$ 
in \eqref{eq10} is bounded. Also inserting the term $K_2$ in \eqref{eq10} one has by integration by
parts that its contribution is null. So in \eqref{eq11} it only remains to consider the contribution from $K_3$ in 
\eqref{eq10}. 
  But using that
  \begin{equation}
  \label{eq15}
  \begin{aligned}
  \|K_3\|_2&=\|\mathcal H(\partial_x w_N^{1+\theta} \partial_x(xu))\|_2=  \|\partial_x w_N^{1+\theta} \partial_x(xu)\|_2\\
 & \leq \|\partial_x w_N^{1+\theta} \,u\|_2 + \|\partial_x w_N^{1+\theta}\,x\, \partial_xu\|_2
\\
&  \leq c( \| w_N^{\theta} \,u\|_2 + \|w_N^{\theta}\,x\, \partial_xu\|_2)=R_1+R_2,
 \end{aligned}
  \end{equation}
  since $R_1$ was previously bounded, it remains to estimate $R_2$ which is equal to the term $F_2$ in 
  \eqref{remain1}. To estimate this term we use the BO equation in \eqref{BO} to obtain the new equation
  \begin{equation}
  \label{neweq}
 \partial_t (x\,\partial_xu) + \mathcal  H\partial_x^2(x\,\partial_x u)-2\mathcal H\partial_x^2 u +x\,\partial_x(u\partial_x u) = 0. 
 \end{equation} 
 
The differential equation \eqref{neweq} multiplied by $w_N^{2\theta}\,x\,\partial_xu$ leads to the identity
 \begin{equation}
\label{eq16}
\begin{aligned}
&\frac{1}{2}\frac{d}{dt}\int{(w_N^{\theta}\,x\partial_x u)^2\,dx}+
\int{w_N^{\theta}\mathcal H \partial_x^2(x\partial_x u) w_N^{\theta} (x\partial_x u)\,dx}\\
&-2\int w_N^{\theta} \mathcal H\partial_x^2u\, w_N^{\theta}x\partial_x u\,dx
+ \int{w_N^{\theta}x\,\partial_x(u\partial_x u)\,w_N^{\theta}x\partial_x u\,dx}=0.
\end{aligned}
\end{equation}
Sobolev inequality and integration by parts lead to 
 \begin{equation}
\label{eq16b}
\begin{aligned}
&|\int{w_N^{\theta}x\,\partial_x(u\partial_x u)\,w_N^{\theta}x\partial_x u\,dx}|\\
&\leq c \|u\|_{2,2}\|w_N^{\theta}x\partial_x u\|_2 ( \|w_N^{\theta}x\partial_x u\|_2 + \|w_N^{\theta} u\|_2),
\end{aligned}
\end{equation}
and since
 \begin{equation}
 \label{eq17}
 \begin{aligned}
&w_N^{\theta} \mathcal H\partial_x^2(x\partial_x u)= -[\mathcal H; w_N^{\theta}] \partial_x^2(x\partial_x u)
 +\mathcal H(w_N^{\theta} \partial_x^2(x\partial_x u))\\
& =V_1 + \mathcal H\partial_x^2(w_N^{\theta} \,x \partial_x u) - 2\mathcal H(\partial_x w_N^{\theta}\,\partial_x(x\partial_xu))
-\mathcal H(\partial_x^2(w_N^{\theta})\,x\,\partial_xu)\\
&=V_1+V_2+V_3+V_4,
\end{aligned}
\end{equation}
 Theorem \ref{theorem8}, the previous results, and interpolation allow to bound the
 $L^2$-norm of the terms $V_1$ and $V_4$. As before by integration by parts the 
 contribution of the term $V_2$ in \eqref{eq16}
 is null. So it just remains to consider the term $V_3$ in \eqref{eq17}. In
 fact,
 $$
 V_3= - 2\mathcal H(\partial_x w_N^{\theta}\,\partial_xu)- 2\mathcal H(\partial_x w_N^{\theta}\,(x\partial^2_xu))
 =V_{3,1}+V_{3,2},
 $$
 so one just needs to handle the term $V_{3,2}$. Using that 
 $$
 |\partial_xw_N^{\theta} \,x|\leq c\, w_N^{\theta},\;\;\;\;c\;\;\text{independent of }\;N,
 $$
 it suffices to consider 
 \begin{equation}
 \label{final1}
 \| \,w_N^{\theta}\,\partial_x^2u\|_2\leq  c \|J^2(w_N^{\theta} u)\|_2 + c \|u\|_{1,2} +c  \| \,w_N^{\theta}\,u\|_2,
 \end{equation}
with $c$ independent on $N$.  So we just need to consider the first term on the r.h.s. of the inequality \eqref{final1}.
 Using interpolation it follows that
 \begin{equation}
 \label{final2}
 \|J^2(w_N^{\theta} u)\|_2 \leq c \|J^{2+\theta}u\|^{2/(2+\theta)}_2\,\| w_N^{2+\theta} u\|^{\theta/(2+\theta)}_s.
 \end{equation}
 We notice that the first term on the r.h.s. of \eqref{final2} is bounded and the second one is bounded by the one 
 we were estimating in \eqref{eq5}. Therefore, \eqref{eq5} and  \eqref{eq16} yield  closed differential inequalities
 for $\,\|xw_N^{1+\theta} \,u\|_2$ and $\,\|w_N^{\theta}\,x\,\partial_x u\|_2$, and consequently the desired result.
 
 \vskip.1in
%%%%%%%%%%%%%%%%%%%%%%
 
  \underline{Case 5: $s=r\in [5/2,7/2)$}.  Part (iii) in Theorem \ref{theorem3}.

 First, by differentiating the BO equation in \eqref{BO} 
  one gets
  $$
  \partial_t (\partial_x u) + \mathcal  H\partial_x^2(\partial_x u) +
  u\partial_x(\partial_x u) +\partial_x u\,\partial_x u = 0,
 $$
 so by reapplying the argument in the previous cases it follows that 
 \begin{equation}
 \label{pastcase}
\sup_{t\in[0,T]} \| \langle x\rangle^{s-1}\,\partial_xu(t)\|_2\leq M,
\end{equation}
with $M$ depending on $\|u_0\|_{s,2},\,\| \langle x\rangle^{s}\,u_0\|_2,$ and $T$.

Next, we multiply the BO equation in \eqref{BO} by $ x^2 w_N^{\tilde \theta}$ with $\tilde \theta \in [1/2,3/2)$
to get
\begin{equation}
\label {111}
\partial_t x^2 w_N^{\tilde \theta}u
 + x^2 w_N^{\tilde \theta}\mathcal  H\partial_x^2u +x^2 w_N^{\tilde \theta} u\partial_x u = 0,
\end{equation}
so a familiar argument  leads to
\begin{equation}
\label {1.1}
\begin{aligned}
&\frac{1}{2}\frac{d\;}{dt}\int (x^2 w_N^{\tilde \theta}u)^2\,dx + \int x^2 w_N^{\tilde \theta}\mathcal  H\partial_x^2u 
\,x^2 w_N^{\tilde \theta}u\,dx \\
&\;\;+\int x^2 w_N^{\tilde \theta} u\partial_x ux^2 w_N^{\tilde \theta}u\,dx = 0.
\end{aligned}
\end{equation}

Using  the identity
 $$
 x^2\,\mathcal H\,\partial_x^2 u =\mathcal H\,\partial_x^2(x^2 \,u) + 4 \mathcal H\,\partial_x(x u) +\mathcal H u,
 $$
  the linear dispersive part of \eqref{111} (the second term on the l.h.s. of \eqref{111}) can be written
 \begin{equation}
 \label{1.3}
 \begin{aligned}
 w_N^{\tilde \theta}\,x^2\,\mathcal H\,\partial_x^2 u &=w_N^{\tilde \theta}\,\mathcal H\,\partial_x^2(x^2 \,u) 
 + 4 w_N^{\tilde \theta}\,\mathcal H\,\partial_x(x u) + w_N^{\tilde \theta}\,\mathcal H u\\
 &=Q_1+Q_2+Q_3.
 \end{aligned}
 \end{equation}
 
Since
   $$
   \int_{-\infty}^{\infty}u_0(x)\,dx=\int_{-\infty}^{\infty}u(x,t)\,dx=0,\;\;\;\;\;\text{then}\;\;\;\;\;H(xu)=xHu,
   $$
  for $\,\tilde \theta\in [1/2,1]$ one has 
   $$
   \|Q_3\|_2=\|w_N^{\tilde \theta} \mathcal H u\|_2\leq \|(1+|x|)\mathcal Hu\|_2
   \leq \|u\|+\|x\,u\|_2,
   $$
   and for $\tilde \theta\in(1,3/2)$ using Proposition \ref{proposition1} 
   $$
   \aligned
   \|Q_3\|_2&=\|w_N^{\tilde \theta} \mathcal H u\|_2\leq \|(1+|x|) w_N^{\tilde \theta-1}\mathcal Hu\|_2
 \\
 &  \leq \|w_N^{\tilde \theta-1}\,u\|+\|\,w_N^{\tilde \theta-1}\,x\,u\|_2,
 \endaligned
   $$
so in both cases by the previous results $Q_3$ in \eqref{1.3} is bounded in $L^2$.

To control $Q_2$ we first consider the case $\tilde \theta\in [1/2,1]$ and use 
Calder\'on commutator theorem to get
$$
 \aligned
 \|Q_2\|_2&=4 \|w_N^{\tilde \theta}\,\mathcal H\,\partial_x(x u)\|_2
\\
& \leq c(\|\,[\mathcal H;  w_N^{\tilde \theta}]\partial_x(x u)\|_2 +
  \| \mathcal H(w_N^{\tilde \theta}\,\partial_x(x u))\|_2)\\
  &\leq c(\|x u\|_2 + \|w_N^{\tilde \theta} \,x\partial_x u\|_2 +\|w_N^{\tilde \theta} \, u\|_2). 
  \endaligned
 $$
Thus, in the case $\,\tilde \theta\in [1/2,1]$,  \eqref{pastcase} provides the appropriate bound on the $L^2$-norm of $Q_2$

For the case  $\,\tilde \theta=1+\theta,\;\theta\in (0,1/2)$ we combine  Proposition \ref{proposition1} and the hypothesis on the mean value of $u_0$
 to deduce that
 $$
 \aligned
 \|Q_2\|_2&=4 \|w_N^{\tilde \theta}\,\mathcal H\,\partial_x(x u)\|_2
 \leq
 c ( \|w_N^{\theta}\,x\,\mathcal H\,\partial_x(x u)\|_2 +
 \|w_N^{\theta}\,\mathcal H\,\partial_x(x u)\|_2)\\
 & 
 \leq c ( \|w_N^{\theta}\,\mathcal H(x\,\partial_x(x u))\|_2 +
 \|w_N^{\theta}\,\partial_x(x u)\|_2)\\
 & \leq c ( \|w_N^{\theta}\,x\,\partial_x(x u)\|_2 +
 \|w_N^{\theta}\,\partial_x(x u)\|_2).
 \endaligned
 $$
  
  Hence, \eqref{pastcase} yields the appropriate bound on the $L^2$-norm of $Q_2$. 
  
  Finally, we turn to the contribution of the term $Q_1$ when inserted in \eqref{1.3}. Thus, we write
  $$
  \aligned
  &w_N^{\tilde \theta}\,\mathcal H\,\partial_x^2(x^2 \,u) = -[\mathcal H; w_N^{\tilde \theta}]\partial^2_x(x^2u)
  +\mathcal H(w_N^{\tilde \theta}\,\partial_x^2(x^2 \,u))\\
  &=V_1 +  \mathcal H(\partial_x^2(w_N^{\tilde \theta}\,x^2 \,u))-2\mathcal H(\partial_x w_N^{\tilde \theta}\,\partial_x(x^2 \,u))-\mathcal H(\partial_x^2w_N^{\tilde \theta}\,(x^2 \,u))
  \\
  &=V_1+V_2+V_3+V_4.
  \endaligned
  $$
  
  From the previous cases it follows that the $L^2$-norm of the terms $V_1,\,V_4$ are bounded.
  By integration by parts, the contribution of the term $V_2$  is null. So it just remains to consider $V_3=
-2\mathcal H(\partial_x w_N^{\tilde \theta}\,\partial_x(x^2 \,u))$ in $L^2$, but
  $$
 \partial_x w_N^{\tilde \theta}\,\partial_x(x^2 \,u)=
  \partial_x w_N^{\tilde \theta}\,(x^2 \,\partial_x u + 2 x\,u)=V_{2,1}+V_{2,2}.
  $$
 Since $\,\tilde \theta\in(1/2,3/2]$
 $$
 \|V_{2,2}\|_2 \leq c\|\langle x\rangle^2u\|_2
 $$
which has been found  to be bounded in the previous cases. Now
since
$$
|\partial_x w_N^{\tilde \theta}\,x^2|\leq \langle x\rangle^{1+\tilde \theta},
$$
it follows that
$$
 \|V_{2,1}\|_2\leq \|\langle x\rangle^{1+\tilde \theta}\,\partial_xu\|_2,
 $$
so \eqref{pastcase} gives the bound. Gathering the above information one completes the proof 
of  Theorem \ref{theorem3}.

 \end{section}
 
%%%%%%%%%
\begin{section}{Proof of Theorem \ref{theorem4}}\label{S: 4}
Without loss of generality we assume that $t_1=0<t_2$.

Since $u(t_1)\in {Z}_{\frac52 ,\frac52}$, we have that $u\in C([0,T]:H^{2+1/2}\cap L^2(|x|^{5^-}dx))$.

Let us denote by $U(t)u_0=(e^{-it|\xi|\xi}\widehat{u_0})^{\vee}$ the solution of the IVP for the linear equation associated to the BO equation  with datum $u_0$. Therefore, the solution to the IVP \eqref{BO} can be represented by Duhamel's formula 
\begin{equation}
  u(t)=U(t)u_0-\int_0^t U(t-t')u(t')\partial_x u(t') dt'.
\label{duh}
\end{equation}

From Plancherel's equality we have that for any $t$, $|x|^{2+1/2}U(t)u_0\in L^2(\R)$ if and only if $D_{\xi}^{1/2}\partial_{\xi}^2(e^{-it|\xi|\xi}\widehat{u}_0)\in L^2(\R)$ and  since
 \begin{equation}
\partial_{\xi}^2(e^{-it|\xi|\xi}\widehat{u}_0)=-e^{-it|\xi|\xi}(4t^2\xi^2\widehat{u}_0+2it\text{sgn}(\xi)\widehat{u}_0+4it|\xi|\partial_{\xi}\widehat{u}_0-\partial_{\xi}^2\widehat{u}_0),
\label{flinear}
\end{equation}
we show that with the hyphotesis on the initial data, all terms in Duhamel's formula for our solution $u$ except the one involving $\text{sgn}(\xi)$, arising from the linear part in \eqref{flinear},  have the appropriate decay at a later time. The argument in our proof requires  localizing near the origin in Fourier frequencies by a function $\chi\in C_0^{\infty}, {\it supp}\,\chi  \subseteq (-\epsilon, \epsilon)$ and $\chi=1$ on $(-\epsilon/2, \epsilon/2)$

Let us start with the computation for the linear part in \eqref{duh} by introducing a commutator as follows
\begin{equation}
\begin{split}
\chi D_{\xi}^{1/2}\partial_{\xi}^2(e^{-it|\xi|\xi}\widehat{u}_0)&=[\chi;D_{\xi}^{1/2}]\partial_{\xi}^2(e^{-it|\xi|\xi}\widehat{u}_0)+D_{\xi}^{1/2}(\chi\partial_{\xi}^2(e^{-it|\xi|\xi}\widehat{u}_0))\\
&=A+B.
\end{split}
\end{equation}

From Proposition \ref{proposition3} and identity \eqref{flinear} we have that
\begin{equation}
\begin{split}
\|A\|_2&= \|[\chi;D_{\xi}^{1/2}]\partial_{\xi}^2(e^{-it|\xi|\xi}\widehat{u}_0)\|_2\\
&\le c\,\|\partial_{\xi}^2(e^{-it|\xi|\xi}\widehat{u}_0)\|_2\\
&\le c\, (t^2\|\xi^2\widehat{u}_0\|_2+t\|\text{sgn}(\xi)\widehat {u}_0\|_2+t\||\xi|\partial_{\xi}\widehat{u}_0\|_2+\|\partial_{\xi}^2\widehat{u}_0\|_2)\\
&\le c\,(t^2\|\partial_2^2{u}_0\|_2+t\|u_0\|_2+t\|\partial_x(xu_0)\|_2+\|x^2u_0\|_2),
\end{split}
\label{2t1}
\end{equation}
which are all finite since $u_0\in Z_{2,2}$.

On the other hand, 
\begin{equation}
\begin{split}
B&=D_{\xi}^{1/2}(\chi\partial_{\xi}^2(e^{-it|\xi|\xi}\widehat{u}_0))\\
&= 4D_{\xi}^{1/2}(\chi e^{-it|\xi|\xi}t^2\xi^2\widehat{u}_0)+2iD_{\xi}^{1/2}(\chi e^{-it|\xi|\xi} t\,\text{sgn} (\xi)\widehat{u}_0)\\
&\,\,\,\,\,+4iD_{\xi}^{1/2}(\chi e^{-it|\xi|\xi}t|\xi|\widehat{u}_0)-D_{\xi}^{1/2}(\chi e^{-it|\xi|\xi}\partial_{\xi}^2\widehat{u}_0)\\
&=B_1+B_2+B_3+B_4.
\end{split}
\label{2t2}
\end{equation}

Next, we shall estimate $B_4$ in $L^2(\R)$. From Theorem \ref{theorem9}, Proposition \ref{propositionB}, and the fractional product rule type inequality \eqref{pointwise-es} we get that

\begin{equation}
\begin{split}
\|B_4\|_2&\le c(\|\chi e^{-it|\xi|\xi}\partial_{\xi}^2\widehat{u}_0\|_2+\|\mathcal D_{\xi}^{1/2}(\chi e^{-it|\xi|\xi}\partial_{\xi}^2\widehat{u}_0)\|_2)\\
&\le c\,(\|\partial_{\xi}^2\widehat{u}_0\|_2+\|\mathcal D_{\xi}^{1/2}( e^{-it|\xi|\xi})\chi\partial_{\xi}^2\widehat{u}_0\|_2+\| e^{-it|\xi|\xi}\mathcal D_{\xi}^{1/2}(\chi\partial_{\xi}^2\widehat{u}_0)\|_2\\
&\le c\,(\|x^2u_0\|_2 +\|(t^{1/4}+t^{1/2}|\xi|^{1/2})\chi\partial_{\xi}^2\widehat{u}_0\|_2+\|\mathcal D_{\xi}^{1/2}(\chi\partial_{\xi}^2\widehat{u}_0)\|_2 )\\
&\le c(T)\,(\|x^2u_0\|_2 +\|\mathcal D_{\xi}^{1/2}(\chi)\|_2\,\|\partial_{\xi}^2\widehat{u}_0\|_{\infty}+\|\chi\|_{\infty}\|\mathcal D_{\xi}^{1/2}(\partial_{\xi}^2\widehat{u}_0)\|_2)\\
&\le c(T)\, \|\langle x\rangle^{2+1/2}u_0\|_2.
\end{split}
\label{b4}
\end{equation}

\vskip.1in
Estimates for $B_1$ and $B_3$ in $L^2(\R)$ are easily obtained in a similar manner involving lower decay and regularity of the initial data. On the other hand for the analysis of $B_2$ we introduce $\tilde{\chi}\in C_0^\infty(\R)$ such that $\tilde{\chi}\equiv 1$ on $ {\it supp}(\chi)$ . Then
we can express this term as
\begin{equation}
\begin{split}
&D_{\xi}^{1/2}(\chi e^{-it|\xi|\xi} t\,\text{sgn} (\xi)\widehat{u}_0)=tD_{\xi}^{1/2}( e^{-it|\xi|\xi} \tilde{\chi}\chi\,\text{sgn} (\xi)\widehat{u}_0)\\
&=t\,([e^{-it|\xi|\xi} \tilde{\chi},D_{\xi}^{1/2}]\chi\,\text{sgn} (\xi)\widehat{u}_0+ e^{-it|\xi|\xi} \tilde{\chi}D_{\xi}^{1/2}(\chi\,\text{sgn} (\xi)\widehat{u}_0))\\
&=t(S_1+S_2).
\end{split}
\label{b21}
\end{equation}

Proposition \ref{proposition3} can be applied to bound $S_1$ in $L^2(\R)$ as 
\begin{equation}
\begin{split}
\|S_1\|_2&\le \|[e^{-it|\xi|\xi} \tilde{\chi},D_{\xi}^{1/2}]\chi\,\text{sgn} (\xi)\widehat{u}_0\|_2\\
&\le c \|\chi\,\text{sgn} (\xi)\widehat{u}_0\|_2\\
&\le c\|u_0\|_2.
\end{split}
\label{b22}
\end{equation}

Therefore, once we show that the integral part in Duhamel`s formula \eqref{duh} lies in $L^{2}(|x|^5 dx)$, we will be able to conclude that 
$$
S_2,\,\;\tilde{\chi}D_{\xi}^{1/2}\chi\,\text{sgn} (\xi)\widehat{u}_0,\,\; D_{\xi}^{1/2}(\tilde{\chi}\chi\,\text{sgn} (\xi)\widehat{u}_0)\in L^2(\R),
$$
 then from Proposition \ref{proposition2}   it will follow that $\widehat{u}_0(0)=0$, and from the conservation law $I_1$ in  \eqref{laws}, this would necessarily imply that  $\widehat{u}_0(0)=\int{u(x,t)dx}=0$.

As  we just mentioned above,  in order to complete the proof, we consider the integral part in Duhamel's formula. We localize again with the help of $\chi\in C_0^{\infty}(\R)$ so that the integral in equation \eqref{duh} after weights and a commutator reads now in Fourier space as
\begin{equation}
\label{ole}
\begin{aligned}
&\int_0^t ([\chi;D_{\xi}^{1/2}] ( e^{-i(t-t')|\xi|\xi}(4(t-t')^2\xi^2\widehat{z}\\
&+2i(t-t')\text{sgn}(\xi)\widehat{z}+4i(t-t')|\xi|\partial_{\xi}\widehat{z}-\partial_{\xi}^2\widehat{z}))\\
&+D_{\xi}^{1/2} (\chi (e^{-i(t-t')|\xi|\xi}(4(t-t')^2\xi^2\widehat{z}\\
&+2i(t-t')\text{sgn}(\xi)\widehat{z}+4i(t-t')|\xi|\partial_{\xi}\widehat{z}-\partial_{\xi}^2\widehat{z}))))\,dt'\\
&=\mathcal{A}_1+\mathcal{A}_2+\mathcal{A}_3+\mathcal{A}_4+\mathcal{B}_1+\mathcal{B}_2+\mathcal{B}_3+\mathcal{B}_4
\end{aligned}
\end{equation}
where $\widehat{z}=\frac{1}{2}\widehat{\partial_x u^2}=i\frac{\xi}{2} \widehat{u}\ast\widehat{u}$.

We limit our attention to the terms in \eqref{ole} involving the highest order derivatives of $u$, i.e. $\mathcal{A}_1$ and $\mathcal{B}_1$, and remark that the others can be treated in a similar way by using that the function $\widehat z$ vanishes at $\xi=0$.
  
Combining Proposition \ref{proposition3}, Holder's inequality and Theorem \ref{theorem10} one has that
\begin{equation}
\begin{split}
\|\mathcal{A}_1\|_{L_T^{\infty}L_x^2}&\le c\,\|(t-t')^2\xi^2\xi \widehat{u}\ast\widehat{u}\|_{L_T^1L_x^2}\\
&\le c(T)\,\|\partial^3_x(uu)\|_{L_T^2L_x^2}\\
&\le c(T)\,(\|u\partial^3_xu\|_{L_T^2L_x^2}+\|\partial_x u\partial^2_xu\|_{L_T^2L_x^2})\\
&\le c(T)\, (\|u\|_{l_k^2L_T^\infty L_x^\infty(Q_k^T)}\|\partial^3_xu\|_{l_k^\infty L_T^2 L_x^2(Q_k^T)}
\\
&\;\;\;\;+\|\partial_x u\|_{L_T^\infty L_x^\infty}\|\partial^2_xu\|_{L_T^\infty L_x^2})\\
&\le c(T)(\|u\|_{l_k^2L_T^\infty L_x^\infty(Q_k^T)}\|\partial^3_xu\|_{l_k^\infty L_T^2 L_x^2(Q_k^T)}
+\|u\|^2_{L_T^\infty H^2}),
\end{split}
\end{equation}
where $\,Q_k^T=[k,k+1]\times [0,T]$. 

For  $\mathcal{B}_1$ we obtain from  Theorem \ref{theorem9} 
\begin{equation}
\label{4.11}
\begin{aligned}
&\|\mathcal{B}_1\|_{L^{\infty}_TL^2_x}\le c  \int_0^T\|D_{\xi}^{1/2}(e^{-i(t-t')|\xi|\xi}\chi\,\xi^2\xi \widehat{u}\ast\widehat{u})\|_2\,dt\\
&\le c (\|e^{-i(t-t')|\xi|\xi}\chi\,\xi^2\xi \widehat{u}\ast\widehat{u}\|_{L^1_TL^2_x}+\|\mathcal D_{\xi}^{1/2}(e^{-i(t-t')|\xi|\xi}\chi\,\xi^2\xi \widehat{u}\ast\widehat{u})\|_{L^1_TL^2_x})\\
&=\mathcal Y_1 +\mathcal Y_2.
\end{aligned}
\end{equation}
These terms can be handled as follows
\begin{equation}
\label{4.12}
\begin{aligned}
\mathcal Y_1\leq c\|\widehat{u}\ast\widehat{u}\|_{L^1_TL^2_x}\leq c \| \|u\|_{\infty} \|u\|_2 \|_{L^1_T}
\leq c T \,\sup_{[0,T]}\|u(t)\|^2_{1,2} ,
\end{aligned}
\end{equation}
and using Proposition \ref{propositionB}, \eqref{pointwise2}, \eqref{pointwise-es}, and \eqref{4.12}
\begin{equation}
\label{4.13}
\begin{aligned}
 \mathcal Y_2 &= \|\mathcal D_{\xi}^{1/2}(e^{-i(t-t')|\xi|\xi}\chi\,\xi^2\xi \widehat{u}\ast\widehat{u})\|_{L^1_TL^2_x}\\
&\leq c \| \mathcal D_{\xi}^{1/2}(e^{-i(t-t')|\xi|\xi})\chi\,\xi^2\xi \widehat{u}\ast\widehat{u}\|_{L^1_TL^2_x}
+ c \|\mathcal D_{\xi}^{1/2}(\chi\,\xi^2\xi \widehat{u}\ast\widehat{u})\|_{L^1_TL^2_x})\\
&\leq c \| (t^{1/2} + t^{1/2}|\xi|^{1/2})\chi\,\xi^2\xi \widehat{u}\ast\widehat{u}\|_{L^1_TL^2_x}
+ c \| \|\mathcal D_{\xi}^{1/2}(\chi\,\xi^3)\|_{\infty} \|\widehat{u}\ast\widehat{u}\|_2\|_{L^1_T}\\
&\,\,\,\,\,+c \|\|\chi\,\xi^3\|_{\infty}\|\mathcal D_{\xi}^{1/2}(\widehat{u}\ast\widehat{u})\|_2\|_{L^1_T}\\
&\leq c(T)\|\widehat{u}\ast\widehat{u}\|_{L^1_TL^2_x} +
c \|\mathcal D_{\xi}^{1/2}(\widehat{u}\ast\widehat{u}\|_{L^1_TL^2_x}\\
& \leq c(T) \,\sup_{[0,T]}\|u(t)\|^2_{1,2} +c(T)\| |x|^{1/2} u\|_{L^{\infty}_TL^2_x}\,\sup_{[0,T]}\|u(t)\|_{1,2}.
\end{aligned}
\end{equation}

Hence the terms in \eqref{ole} are all bounded, so by applying the argument after inequality \eqref{b22} we complete the proof.

\end{section}

%%%%%%%%%%%%%%%%%%%%%%%%% 

%%%%%%%%%
\begin{section}{Proof of Theorem \ref{theorem5}}\label{S: 5}

From the previous results and the hypothesis we have that for any $\epsilon>0$
$$
u\in C([0,T] :\dot Z_{7/2,7/2-\epsilon})\;\;\;\text{and}\;\;\;u(\cdot,t_j)\in H^{7/2}(\R),\;\;j=1,2,3.
$$
Hence,
$$
\widehat u\in C([0,T] : H^{7/2-\epsilon}(\mathbb R)\cap L^2(|\xi|^7d\xi))\;\;\;\text{and}\;\;\;\widehat u(\cdot,t_j)\in L^2(|\xi|^7d\xi),\;\;j=1,2,3.
$$
for any $\epsilon>0$. Thus, in particular it follows that
\begin{equation}
\label{123}
\widehat u \ast \widehat u\in  C([0,T] : H^{6}(\mathbb R)\cap L^2(|\xi|^7d\xi)).
\end{equation}

Let us assume that $t_1=0<t_2<t_3$. An explicit computation shows that
\begin{equation}
 \begin{split}
F(t,\xi,\widehat{u}_0)&=\partial_{\xi}^3(e^{-it|\xi|\xi}\widehat{u}_0)\\
&=e^{-it|\xi|\xi}(8it^3\xi^3\widehat{u}_0-12t^{2}\xi\widehat{u}_0-12t^{2}\xi^2\partial_{\xi} \widehat{u}_0\\
&\,\,\,\,\,-6it\,\text{sgn}(\xi)\partial_{\xi} \widehat{u}_0-6it|\xi|\partial_{\xi} ^2\widehat{u}_0-2it\delta\widehat{u}_0+\partial_{\xi} ^3\widehat{u}_0),
\end{split}
\label{flinear3}
\end{equation}
where we observe that  since the initial data $u_0$ has zero mean value  the term involving the Dirac function in  \eqref{flinear3} vanishes.  Hence in order to prove our theorem, via Plancherel's Theorem and Duhamel`s formula \eqref{duh}, it is enough to show that the assumption that 
 \begin{equation}
D_{\xi}^{1/2}F(t,\xi,\widehat{u}_0)-\int_0 ^tD_{\xi}^{1/2}F(t-t',\xi,\widehat{z}(t'))\,dt',
\label{ec}
\end{equation}
lies in $L^2(\R)$ for times $t_1=0<t_2<t_3$, where $\widehat{z}= \frac{1}{2}\widehat {\partial_x u^2}=i \frac{\xi}{2} \widehat{u}\ast\widehat{u}$, leads to a contradiction.
Let us show that the first term in equation \eqref{ec} which arises from the linear part in Duhamel's formula persists in $L^2$

We proceed as in the proof of Theorem \ref{theorem4} and  localize one more time  by introducing $\chi\in C_0^{\infty}, {\it supp}\,\chi  \subseteq (-\epsilon, \epsilon)$ and $\chi=1$ on $(-\epsilon/2, \epsilon/2)$ so that

\begin{equation}
\begin{split}
\chi\,D_{\xi}^{1/2}\partial_{\xi} ^3(e^{-it|\xi|\xi}\widehat{u}_0)&=[\chi;D_{\xi}^{1/2}]\partial_{\xi} ^3(e^{-it|\xi|\xi}\widehat{u}_0)+D_{\xi}^{1/2}(\chi\partial_{\xi} ^3(e^{-it|\xi|\xi}\widehat{u}_0))\\
&=\acont+\bcont.
\end{split}
\end{equation}
 
As for the first term, $\acont$,  from Proposition \ref{proposition3}, this is bounded in $L^2(\R)$ by $\|\partial_{\xi} ^3(e^{-it|\xi|\xi}\widehat{u}_0)\|_2$, which is finite as can easily be observed from its explicit representation in \eqref{flinear3}, the assumption on the initial data $u_0$, and the quite similar computation already performed in \eqref{2t1}, therefore we omit the details.

On the other hand, for $\bcont$, we notice that 
\begin{equation}
\begin{split}
\bcont &=D_{\xi}^{1/2}(\chi\partial_{\xi} ^3(e^{-it|\xi|\xi}\widehat{u}_0))\\
&= 8iD_{\xi}^{1/2}(\chi e^{-it|\xi|\xi}t^3|\xi|^3\widehat{u}_0)-12D_{\xi}^{1/2}(\chi e^{-it|\xi|\xi} t^2\,\xi\widehat{u}_0)\\
&\,\,\,\,\,-12D_{\xi}^{1/2}(\chi e^{-it|\xi|\xi} t^2\,\xi^2\partial_{\xi} \widehat{u}_0)-6iD_{\xi}^{1/2}(\chi e^{-it|\xi|\xi} t\,\text{sgn} (\xi)\partial_{\xi}  \widehat{u}_0)\\
&\,\,\,\,\,-6iD_{\xi}^{1/2}(\chi e^{-it|\xi|\xi}t|\xi|\partial_{\xi} ^2\widehat{u}_0) +D_{\xi}^{1/2}(\chi e^{-it|\xi|\xi}\partial_{\xi} ^3\widehat{u}_0)\\
&=\bcont_1+\bcont_2+\bcont_3+\bcont_4+\bcont_5+\bcont_7.
\end{split}
\label{4t2}
\end{equation}

Notice that  from the remark made after the identity \eqref{flinear3} $\bcont_6$ does not appear,  and that $\bcont_1$ and $\bcont_7$ are the terms involving the  highest regularity and decay of the initial data. Therefore we show in detail their $L^2$ estimates  along with the argument to exploit a  nice cancellation property of $\bcont_4,$ and a term arising in the integral part in Duhamel's formula \eqref{duh} .

For $\bcont_1$ we obtain from Theorem \ref{theorem9}, fractional product rule type estimate \eqref{pointwise2} ,
\eqref{pointwise-es}, and Holder's inequality  that
\begin{equation}
\begin{split}
\|\bcont _1\|_2&\le c(\|\chi e^{-it|\xi|\xi}t^3\xi^3\widehat{u}_0\|_2+\|\mathcal D_{\xi}^{1/2}(\chi e^{-it|\xi|\xi}t^3|\xi|^3\widehat{u}_0)\|_2)\\
&\le c\,t^3(\|u_0\|_2+\|\mathcal D_{\xi}^{1/2}( e^{-it|\xi|\xi})\chi|\xi|^3\widehat{u}_0\|_2+\| e^{-it|\xi|\xi}\mathcal D_{\xi}^{1/2}(\chi |\xi|^3\widehat{u}_0)\|_2)\\
&\le c\,t^3(\|u_0\|_2 +\|(t^{1/4}+t^{1/2}|\xi|^{1/2})\chi|\xi|^3\widehat{u}_0\|_2+\|\mathcal D_{\xi}^{1/2}(\chi|\xi|^3\widehat{u}_0)\|_2 )\\
&\le c(T)\,(\|u_0\|_2 +\|\mathcal D_{\xi}^{1/2}(\chi\xi^3)\|_{\infty}\,\|\widehat{u}_0\|_2+\|\chi\xi^3\|_{\infty}\|\mathcal D_{\xi}^{1/2} \widehat{u}_0\|_2)\\
&\le c(T)\,(\|u_0\|_2 +\| |x|^{1/2} u_0\|_2),
\end{split}
\label{bc1}
\end{equation}
and similarly
\begin{equation}
\begin{split}
\|\bcont_7\|_2&\le c(\|\chi e^{-it|\xi|\xi}\partial_{\xi}^3\widehat{u}_0\|_2+\|\mathcal D_{\xi}^{1/2}(\chi e^{-it|\xi|\xi}\partial_{\xi}^3\widehat{u}_0)\|_2)\\
&\le c\,(\|\partial_{\xi}^3\widehat{u}_0\|_2+\|\mathcal D_{\xi}^{1/2}( e^{-it|\xi|\xi})\chi\partial_{\xi}^3\widehat{u}_0\|_2+\| e^{-it|\xi|\xi}\mathcal D_{\xi}^{1/2}(\chi\partial_{\xi}^3\widehat{u}_0)\|_2\\
&\le c\,(\|x^3u_0\|_2 +\|(t^{1/4}+t^{1/2}|\xi|^{1/2})\chi\partial_{\xi}^3\widehat{u}_0\|_2+\|\mathcal D_{\xi}^{1/2}(\chi\partial_{\xi}^3\widehat{u}_0)\|_2 )\\
&\le c(t)\,(\|x^3u_0\|_2 +\|\mathcal D_{\xi}^{1/2}(\chi)\|_{\infty}\,\|\partial_{\xi}^3\widehat{u}_0\|_{2}+\|\chi\|_{L^\infty}\|\mathcal D_{\xi}^{1/2}(\partial_{\xi}^3\widehat{u}_0)\|_2)\\
&\le c(T)\,(\|x^3u_0\|_2 +\|D_{\xi}^{1/2}(\partial_{\xi}^3\widehat{u}_0)\|_2)\\
&\le c(T)\, \|\langle x\rangle^{3+1/2}u_0\|_2.
\end{split}
\label{bc7}
\end{equation}

Now, let us go over the integral part that can be written in Fourier space and with the help of a commutator  as
\begin{equation}
 \begin{split}
&\int_0^t ([\chi;D_{\xi}^{1/2}](e^{-i(t-t')|\xi|\xi}(8i(t-t')^3\xi^3\widehat{z}-12(t-t')^{2}\xi\widehat{z}\\
&-12(t-t')^{2}\xi^2\partial_{\xi} \widehat{z}-6i(t-t')|\xi|\partial_{\xi} ^2\widehat{z}-2i(t-t')\delta\widehat{z}
+\partial_{\xi} ^3\widehat{z}))\\
&+D_{\xi}^{1/2}(\chi(8i(t-t')^3\xi^3\widehat{z}-12(t-t')^{2}\xi\widehat{z}\\
&-12(t-t')^{2}\xi^2\partial_{\xi} \widehat{z}-6i(t-t')|\xi|\partial_{\xi} ^2\widehat{z}-2i(t-t')\delta\widehat{z}+\partial_{\xi} ^3\widehat{z})))dt\\
&=\widetilde{\mathcal{A}}_1+\widetilde{\mathcal{A}}_2+\widetilde{\mathcal{A}}_3+\widetilde{\mathcal{A}}_5
+\widetilde{\mathcal{A}}_6+\widetilde{\mathcal{A}}_7+\widetilde{\mathcal{B}}_1+\widetilde{\mathcal{B}}_2+\widetilde{\mathcal{B}}_3
+\widetilde{\mathcal{B}}_5+\widetilde{\mathcal{B}}_6+\widetilde{\mathcal{B}}_7+\widetilde{\mathcal{C}},
\label{ic1}
\end{split}
\end{equation}
where
\begin{equation}
\widetilde{\mathcal{C}}=-6i\int_0^t D_{\xi}^{1/2}(e^{-i(t-t')|\xi|\xi}\chi\,(t-t')\text{sgn}(\xi)\partial_{\xi} \widehat{z})\,dt',
\end{equation}
and  $\widehat{z}=\frac{1}{2}\widehat{\partial_x u^2}=i\frac{\xi}{2} \widehat{u}\ast\widehat{u}$.

Notice that $\widetilde{\mathcal{A}}_6,\widetilde{\mathcal{B}}_6 $ vanish since $u\partial_x u $ has zero mean value and for 
$\widetilde{\mathcal{A}}_1, \widetilde{\mathcal{A}}_2, \widetilde{\mathcal{A}}_3, \widetilde{\mathcal{A}}_5$, $\widetilde{\mathcal{A}}_7, \widetilde{\mathcal{B}}_1,\widetilde{\mathcal{B}}_2,\widetilde{\mathcal{B}}_3,\widetilde{\mathcal{B}}_5$ and $\widetilde{\mathcal{B}}_7$ the estimates in $L^2(\R)$ are essentially the same for their counterparts   in equation \eqref{ole}, in the proof of Theorem \ref{theorem4}, so we omit the details of their estimates.

Therefore from the assumption that $u_0,u(t_2)\in \dot{Z}_{\frac72 ,\frac{7}{2}}$, equation \eqref{ic1}, and the estimates above, we  conclude 
that 
\begin{equation}
\begin{split}
R&=-6iD_{\xi}^{1/2}(e^{-it|\xi|\xi}\chi\, t\,\text{sgn}(\xi)\partial_{\xi} \widehat{u}_0)-\widetilde{\mathcal{C}}\\
&=-6iD_{\xi}^{1/2}(e^{-it|\xi|\xi}\chi\, t\,\text{sgn}(\xi)\partial_{\xi} \widehat{u}_0)\\
&\,\,\,\,\,+6i\int_0^t D_{\xi}^{1/2}(e^{-i(t-t')\,|\xi|\xi}\chi\,(t-t')\text{sgn}(\xi) \partial_{\xi} (\frac{i \xi}{2} \widehat{u}\ast\widehat{u}))\,dt',
\end{split}
\label{last}
\end{equation}
is a function in $L^2(\R)$ at time $t=t_2$. But
\begin{equation}
\begin{split}
R&=6i\int_0^t (D_{\xi}^{1/2}(e^{-i(t-t')\,|\xi|\xi}\chi\,(t-t')\text{sgn}(\xi)\\
&\,\,\,\,\,\,(\partial_{\xi} (\frac{i\xi}{2} \widehat{u}\ast\widehat{u})-\partial_{\xi} (\frac{i\xi}{2} \widehat{u}\ast\widehat{u})(0)))\,dt'\\
&\,\,\,\,\,-6iD_{\xi}^{1/2}(e^{-it|\xi|\xi}\chi\, t\,\text{sgn}(\xi)(\partial_{\xi} \widehat{u}_0-\partial_{\xi}  \widehat{u}_0(0)))\\
&\,\,\,\,\,-6iD_{\xi}^{1/2}(e^{-it|\xi|\xi}\chi\, t\,\text{sgn}(\xi)\partial_{\xi}  \widehat{u}_0(0))\\
&\,\,\,\,\,+6i\int_0^t D_{\xi}^{1/2}(e^{-i(t-t')\,|\xi|\xi}\chi\,(t-t')\text{sgn}(\xi)(\partial_{\xi} (\frac{i\xi}{2} \widehat{u}\ast\widehat{u}(0)))\,dt'\\
&=R_1+R_2+R_3+R_4.
\end{split}
\label{last1}
\end{equation}

We shall show that  $R_1$ and $R_2$ are $L^2(\R)$ functions. This will imply that  $(R_3+R_4)(t_2)$ is also an $L^2(\R)$ function. 

For $R_1$ we observe that from \eqref{123}
$$
e^{-i(t-t')|\xi|\xi}\chi(\xi) \text{sgn}(\xi) (\partial_{\xi} (\frac{i\xi}{2} \widehat{u}\ast\widehat{u})(\xi,t')-\partial_{\xi} (\frac{i\xi}{2} \widehat{u}\ast\widehat{u})(0,t')))
$$
is a Lipschitz function with compact support in the $\xi$ variable. Therefore, using Theorem \ref{St1} one sees that
$ R_1(t)\in L^2(\R)$. A similar argument shows that $ R_2(t)\in L^6(\R)$. Therefore, we have that $(R_3+R_4)(t_2)\in L^2(\R)$.

On the other hand 
$$
\partial_{\xi} (\frac{i\xi}{2} \widehat{u}\ast\widehat{u})(0)=\widehat{-ixu\partial_x u\,}(0)=-i\int{xu\partial_x u\,dx}=\frac{i}{2}\|u\|^2_2,
$$ 
and from the Benjamin-Ono equation we have
\begin{equation}
\frac{d}{dt}\int{xu}dx+\int{x\partial^2_x\mathcal H u}dx+\int{xu\partial_x u}dx=0,
\end{equation}
which implies that
\begin{equation}
\frac{d}{dt}\int{xu}dx=-\int{xu\partial_x u}dx=\frac{1}{2}\|u_0\|^2_2,
\end{equation}
and hence 
$$
\partial_{\xi}(\frac{i\xi}{2} \widehat{u}\ast\widehat{u})(0)=i\frac{d}{dt}\int{xu}dx.
$$

Substituting this into $R_4$ gives us after integration by parts 
\begin{equation}
\begin{split}
R_4&=-6\int_0^t D_{\xi}^{1/2}(e^{-i(t-t')\,|\xi|\xi}\chi\,(t-t')\text{sgn}(\xi)(\frac{d}{dt'}\int{xu}dx))\,dt'\\
&= -6D_{\xi}^{1/2}(e^{-i(t-t')\,|\xi|\xi}\chi\,(t-t')\text{sgn}(\xi)\int{xu}dx |^{t'=t}_{t'=0})\\
&+6\int_0^t D_{\xi}^{1/2}(e^{-i(t-t')\,|\xi|\xi}\chi\,(i|\xi|\xi(t-t')-1)\text{sgn}(\xi)(\int{xu}dx))\,dt'\\
&=6D_{\xi}^{1/2}(e^{-it\,|\xi|\xi}\chi\,t\,\text{sgn}(\xi)\int{xu_0(x)}dx)\\
&+6i\int_0^t D_{\xi}^{1/2}(e^{-i(t-t')\,|\xi|\xi}\chi\,(t-t')|\xi|\xi\,\text{sgn}(\xi)(\int{xu}dx))\,dt'\\
&-6\int_0^t D_{\xi}^{1/2}(e^{-i(t-t')\,|\xi|\xi}\chi\,\text{sgn}(\xi)(\int{xu}dx))\,dt'.
\end{split}
\label{last2}
\end{equation}
We observe that the second term after the last equality in \eqref{last2} belongs to $L^2(\R)$ 
and the first cancels out with $R_3$ since 
\begin{equation}
\begin{split}
\partial_{\xi}  \widehat{u}_0(0)=-i\widehat{xu_0}(0)=-i\int{xu_0(x)}\,dx,
\end{split}
\end{equation}
and therefore 
\begin{equation}
R_3=-6D_{\xi}^{1/2}(e^{-it|\xi|\xi}\chi\, t\,\text{sgn}(\xi)\int{xu_0(x)}dx).
\end{equation}

So the argument above implies that 
\begin{equation}
-6\int_0^t D_{\xi}^{1/2}(e^{-i(t-t')\,|\xi|\xi}\chi\,\text{sgn}(\xi)(\int{xu(x,t')}dx))\,dt'
\label{last3}
\end{equation}
is in $L^2(\R)$ at time $t=t_2$,  and from Theorem \ref{theorem9} this is equivalent to have that
\begin{equation}
\mathcal{D}_{\xi}^{1/2}(\chi(\xi) \,\text{sgn}(\xi)\, \int_0^{t_2} e^{-i(t_1-t')\,|\xi|\xi}(\int{xu(x,t')}dx)\,dt')\,\in L^2(\R),
\label{last4}
\end{equation}
which from Proposition \ref{proposition2} (choosing the support $(-\epsilon,\epsilon)$ of $\chi$ sufficiently small) implies that $ \int_0^{t_2} (\int{xu(x,t')}dx)\,dt'\,=0$ and consequently   $\int xu(x,t)dx$ must be zero at some time in $(0,t_2)$. We reapply the same argument to conclude that $\int xu(x,t)dx$ is again zero at some other time in $(t_2,t_3)$. Finally, the identity \eqref{ole1} completes the proof of the theorem.
\end{section}

\vskip.2in
{\underbar{ACKNOWLEDGMENT}}  : The authors would like to thank J. Duoandikoetxea for fruitful conversations concerning this paper. 
This work was done while G. F. was visiting the Department of Mathematics at the University of California-Santa Barbara whose hospitality he gratefully acknowledges. G.P. was supported by NSF grant DMS-0800967.

\end{document}